\documentclass[12pt]{article}
\usepackage{amssymb}
\usepackage{amsthm}
\usepackage{amsmath}
\usepackage{graphicx}
\usepackage{enumerate}
\usepackage{pict2e}
\usepackage{framed}

\DeclareMathOperator{\Id}{Id}

\DeclareMathOperator{\BId}{\mathbf{Id}}

\DeclareMathOperator{\Max}{Max}

\newtheorem{theorem}{Theorem}%[section]

\newtheorem{lemma}[theorem]{Lemma}
\newtheorem{proposition}[theorem]{Proposition}

\newtheorem{example}[theorem]{Example}
\newtheorem{corollary}[theorem]{Corollary}
\title{Orthogonal adjointness in posets with $0$ \thanks{Support of the research of the first and second author by the Czech Science Foundation (GA\v CR), project 24-14386L, and support of the research of the third author by the Austrian Science Fund (FWF), project 10.55776/PIN5424624, is gratefully acknowledged.}}
\author{Michal~Botur, Ivan~Chajda and Helmut~L\"anger$^1$}
\date{}

\begin{document}
	
\maketitle

\footnotetext[1]{Corresponding author}
	
\begin{abstract}
Motivated by the concept of polarity introduced by G.~Birkhoff for a binary relation on a set, we introduce a concept of orthogonality in a poset with $0$. A pair of operators $f$, $g$ on a poset with $0$ is called orthogonally adjoint if $f(x)$ is orthogonal to $y$ if and only if $x$ is orthogonal to $g(y)$. We characterize the existence and the uniqueness of $g$ for given $f$ and describe basic properties of orthogonal adjointness. We present constructions of orthogonally adjoint pairs in pseudocomplemented posets. If a given operator $f$ is an order-isomorphism of a pseudocomplemented poset satisfying some natural properties then the corresponding adjoint $g$ can be described explicitly. Moreover, if $f$ and $f^{-1}$ are bijective $\perp$-morphisms then they are orthogonally adjoint, too. Finally we show that a given pair of orthogonally adjoint mappings on a poset $\mathbf P$ may not be extendable to the Dedekind-McNeille completion of $\mathbf P$ and we present sufficient conditions for the existence of such an extension. We also provide sufficient conditions for the existence of an extension of orthogonally adjoint mappings to the lattice of ideals. Our results are illustrated by numerous examples.	
\end{abstract}



The concept of adjointness was originally introduced by G.~Birkhoff \cite B under the name {\em polarity} as follows: Let $\mathbf P=(P,\le)$ be a poset, $\rho$ be a binary relation on $P$ and $f,g\colon P\to P$. Then $(f,g)$ is called an {\em adjoint pair} of $\mathbf P$ if
\[
f(x)\mathrel{\rho}y\text{ if and only if }x\mathrel{\rho}g(y)
\]
for all $x,y\in P$. This concept was treated by many authors in lattices and semilattices where the relation $\rho$ is usually the induced order $\le$. However, recently J.~Paseka, J.K.~Putra and R.~Smolka \cite{PPS} generalized this concept to so-called dynamic algebras with the help of which certain tense logics can be formalized. The mentioned authors use a kind of adjointness where the relation $\rho$ is orthogonality in an orthomodular lattice. The present authors together with M.~Kola\v r\'ik introduced already the concept of orthogonality in posets with $0$ (see e.g.\ \cite{CKL}). This motivated us to define the so-called {\em orthogonal adjointness} in posets with $0$, in particular in pseudocomplemented ones. It is worth noticing that adjointness of mappings and homomorphisms in posets has been investigated by the authors already in \cite{CL}.

In the following we often identify singletons with their unique element.

Let $\mathbf P=(P,\le)$ be a poset, $a,b\in P$ and $A,B\subseteq P$. We define
\begin{align*}
& A\le B\text{ if }x\le y\text{ for all }x\in A\text{ and all }y\in B, \\
& L(A):=\{x\in P\mid x\le A\}, \\
& U(A):=\{x\in P\mid A\le x\}.
\end{align*}
Instead of $L(\{a\})$, $L(\{a,b\})$, $L(A\cup\{a\})$, $L(A\cup B)$ and $L\big(U(A)\big)$ we simply write $L(a)$, $L(a,b)$, $L(A,a)$, $L(A,B)$ and $LU(A)$, respectively. Analogously we proceed in similar cases. Put $D(\mathbf P):=\{L(C)\mid C\subseteq P\}$. Then $\mathbf D(\mathbf P):=\big(D(\mathbf P),\subseteq\big)$ is a complete lattice, called the {\em Dedekind-MacNeille completion} of $\mathbf P$. The mapping $x\mapsto L(x)$ is an embedding of $\mathbf P$ into $\mathbf D(\mathbf P)$. $\mathbf D(\mathbf P)$ is in some sense the smallest extension of $\mathbf P$ to a complete lattice. It preserves existing joins and meets, and $P$ is join- and meet-dense in $\mathbf D(\mathbf P)$.

Now let $(P,\le,0)$ be a poset with $0$. We define
\begin{align*}
& A\perp B\text{ if }x\wedge y=0\text{ for all }x\in A\text{ and all }y\in B, \\	
& A^\perp:=\{x\in P\mid x\perp A\}, \\
& \Max A:=\text{set of all maximal elements of }(A,\le). 
\end{align*}
A mapping from $P$ to $P$ will be called a {\em mapping in $P$}. Let $M(P)$ denote the set of all mappings in $P$. A mapping from $P$ to $2^P$ will be called an {\em operator on $P$}. Let $Op(P)$ denote the set of all operators on $P$. $M(P)$ can be considered as a subset of $Op(P)$.

On $Op(P)$ define a binary relation $\le$ by
\[
f\le g\text{ if }f(x)\subseteq g(x)\text{ for all }x\in P
\]
for all $f,g\in Op(P)$. Then $(Op(P),\le)$ is a poset. For all $f\in Op(P)$ let $\Max f\in Op(P)$ be defined by $(\Max f)(x):=\Max\big(f(x)\big)$ for all $x\in P$.

Assume $f,g\in Op(P)$. We say that $f$ and $g$ are {\em orthogonally adjoint} if
\[
f(x)\perp y\text{ if and only if }x\perp g(y)
\]
for all $x,y\in P$. This is equivalent to
\[
f(x)\subseteq y^\perp\text{ if and only if }g(y)\subseteq x^\perp
\]
for all $x,y\in P$ or to
\[
y\subseteq\big(f(x)\big)^\perp\text{ if and only if }x\subseteq\big(g(y)\big)^\perp\tag{1}
\]
for all $x,y\in P$. Obviously, the constant mapping $0$ with value $0$ is orthogonally adjoint to itself. Moreover, if $f\in M(P)$ is orthogonally adjoint to $0$ then for all $x\in P$ we have $x\wedge0\big(f(x)\big)=0$ and hence $f(x)=f(x)\wedge f(x)=0$, i.e.\ $f=0$.

Let $f\in M(P)$ and $g\in Op(P)$. Then we define
\begin{align*}
f^{-1}(A) & :=\{x\in P\mid f(x)\in A\}, \\
g^{-1}(A) & :=\{x\in P\mid g(x)\subseteq A\}.
\end{align*}

We are going to characterize when in a poset $(P,\le,0)$ with $0$ for a given mapping $f$ in $P$ there exists a (unique) mapping $g$ in $P$ being orthogonally adjoint to $f$.

\begin{lemma}
Let $(P,\le,0)$ be a poset with $0$ and $f,g\in Op(P)$. Then the following are equivalent:
\begin{enumerate}[{\rm(i)}]
\item $f$ and $g$ are orthogonally adjoint,
\item for all $x\in P$ we have $f^{-1}(x^\perp)=\big(g(x)\big)^\perp$.
\end{enumerate}
\end{lemma}

\begin{proof}
For all $x,y\in P$ the following are equivalent:
\[
f(y)\wedge x=0,\enspace f(y)\subseteq x^\perp,\enspace y\in f^{-1}(x^\perp)
\]
and the following are equivalent.
\[
y\wedge g(x)=0,\enspace y\in\big(g(x)\big)^\perp.
\]
\end{proof}

From this lemma we immediately obtain the following theorem:

\begin{theorem}
Let $(P,\le,0)$ be a poset with $0$ and $f\in Op(P)$. Then the following hold:
\begin{enumerate}[{\rm(i)}]
\item There exists some $g\in Op(P)$ being orthogonally adjoint to $f$ if and only if for every $x\in P$ there exists some $y\in P$ with $f^{-1}(x^\perp)=y^\perp$,
\item there exists a unique $g\in Op(P)$ being orthogonally adjoint to $f$ if and only if for every $x\in P$ there exists a unique $y\in P$ with $f^{-1}(x^\perp)=y^\perp$.
\end{enumerate}
\end{theorem}

Recall that a poset satisfies the {\em Ascending Chain Condition} if it has no infinite ascending chains. Assume the poset $(P,\le)$ to have this condition. Then, whenever $a\in A\subseteq P$, there exists some $b\in\Max A$ satisfying $a\le b$. This shows that for any non-empty subset $A$ of $P$ the set $\Max A$ is not empty, too. Observe that any finite poset satisfies the Ascending Chain Condition.

\begin{lemma}\label{lem5}
Let $(P,\le,0)$ be a poset with $0$ satisfying the Ascending Chain Condition and $A\subseteq P$. Then $(\Max A)^\perp=A^\perp$.
\end{lemma}

\begin{proof}
Because of $\Max A\subseteq A$ we have $A^\perp\subseteq(\Max A)^\perp$. Conversely, assume $b\in(\Max A)^\perp$. If $a\in A$ then according to the Ascending Chain Condition there exists some $c\in\Max A$ with $a\le c$ and hence $b\wedge c=0$ which implies $b\wedge a=0$. Therefore $b\in A^\perp$ showing $(\Max A)^\perp\subseteq A^\perp$.
\end{proof}

Some elementary properties of orthogonally adjoint operators are listed in the following lemma. Within this lemma and its proof we write $f\sim g$ whenever $f$ and $g$ are orthogonally adjoint.

\begin{lemma}\label{lem6}
Let $\mathbf P=(P,\le,0)$ be a poset with $0$, $f,g,h,F,G\in Op(P)$ and $n$ a positive integer. Then the following hold:
\begin{enumerate}[{\rm(i)}]
\item $f\sim g$ implies $g\sim f$,
\item $f\sim g\sim h\sim F$ implies $f\sim F$,
\item $f\sim g$ and $F\sim G$ imply $f\circ F\sim G\circ g$,
\item $f\sim g$ implies $f^n\sim g^n$,
\item $f\le g\le h$, $f\sim F$ and $h\sim F$ imply $g\sim F$,
\item if $\mathbf P$ satisfies the Ascending Chain Condition then $f\sim g$ is equivalent to $(\Max f)\sim g$,
\item if $f\in Op(P)$ is defined by $f(x):=x^{\perp\perp}$ for all $x\in P$ then $f$ is orthogonally adjoint to itself.
\end{enumerate}
\end{lemma}

\begin{proof}
\begin{enumerate}[(i)]
\item This follows directly from the definition of orthogonal adjointness.
\item If $f\sim g\sim h\sim F$ then for $a,b\in P$ the following are equivalent:
\[
f(a)\wedge b=0,\enspace a\wedge g(b)=0,\enspace g(b)\wedge a=0,\enspace b\wedge h(a)=0,\enspace h(a)\wedge b=0,\enspace a\wedge F(b)=0.
\]
\item If $f\sim g$ and $F\sim G$ then for $a,b\in P$ the following are equivalent:
\begin{align*}
(f\circ F)(a)\wedge b=0, & \enspace f\big(F(a)\big)\wedge b=0,\enspace F(a)\wedge g(b)=0,\enspace a\wedge G\big(g(b)\big)=0, \\
                         & \enspace a\wedge(G\circ g)(b)=0.
\end{align*}
\item This follows from (iii) by induction on $n$.
\item If $f\le g\le h$, $f\sim F$ and $h\sim F$ then for $a,b\in P$ any of the following statements implies the next one:
\[
g(a)\wedge b=0,\enspace f(a)\wedge b=0,\enspace a\wedge F(b)=0,\enspace h(a)\wedge b=0,\enspace g(a)\wedge b=0.
\]
\item This follows from (1), Lemma~\ref{lem5} and $\big((\Max f)(x)\big)^\perp=\Big(\Max\big(f(x)\big)\Big)^\perp=\big(f(x)\big)^\perp$ for all $x\in P$.
\item For $a,b\in P$ we have $a^{\perp\perp}\subseteq b^\perp$ if and only if $b^{\perp\perp}\subseteq a^\perp$.
\end{enumerate}
\end{proof}

It is worth noticing that the relation of orthogonal adjointness is not transitive and hence the premise in assertion (ii) of Lemma~\ref{lem6} makes sense.

In the following we present a method for generating non-trivial pairs of orthogonally adjoint operators.

Let $(P,\le,0)$ be a poset with $0$. We call a {\em subset} $A$ of $P$ {\em $\perp$-closed} if $A^{\perp\perp}=A$. This is equivalent to the fact that there exists some subset $B$ of $P$ with $B^\perp=A$. Now we can state the following:

\begin{theorem}\label{th1}
Let $(P,\le,0)$ be a poset with $0$, let $a\in P$ and $A$ and $B$ be $\perp$-closed subsets of $P$. Construct $f,g\in Op(P)$ as follows: If $a\in A$ then put $f(a):=0$. If $a\notin A$ then let $f(a)$ be a subset of $P$ satisfying $\big(f(a)\big)^\perp=B$. If $a\in B$ then put $g(a):=0$. If $a\notin B$ then let $g(a)$ be a subset of $P$ satisfying $\big(g(a)\big)^\perp=A$. Then $f$ and $g$ are orthogonally adjoint.
\end{theorem}

\begin{proof}
For $a,b\in P$ the following are equivalent: \\
$f(a)\wedge b=0$, \\	
$a\in A$ or $\Big(a\notin A$ and $b\in\big(f(a)\big)^\perp\Big)$, \\
$a\in A$ or $(a\notin A$ and $b\in B)$, \\
$a\in A$ or $b\in B$, \\
$b\in B$ or $a\in A$, \\
$b\in B$ or $(b\notin B$ and $a\in A)$, \\
$b\in B$ or $\Big(b\notin B$ and $a\in\big(g(b)\big)^\perp\Big)$, \\
$a\wedge g(b)=0$.
\end{proof}

\begin{example}\label{ex2}
Consider the bounded poset $(P,\le,0,1)$ visualized in Fig.~1:	
	
\vspace*{-3mm}
	
\begin{center}
\setlength{\unitlength}{7mm}
\begin{picture}(6,8)
\put(3,1){\circle*{.3}}
\put(1,3){\circle*{.3}}
\put(5,3){\circle*{.3}}
\put(3,4){\circle*{.3}}
\put(1,5){\circle*{.3}}
\put(5,5){\circle*{.3}}
\put(3,7){\circle*{.3}}
\put(3,1){\line(-1,1)2}
\put(3,1){\line(0,1)6}
\put(3,1){\line(1,1)2}
\put(3,7){\line(-1,-1)2}
\put(3,7){\line(1,-1)2}
\put(1,3){\line(0,1)2}
\put(5,3){\line(0,1)2}
\put(2.85,.3){$0$}
\put(.35,2.85){$a$}
\put(5.4,2.85){$b$}
\put(.35,4.85){$d$}
\put(5.4,4.85){$e$}
\put(3.4,3.85){$c$}
\put(2.85,7.4){$1$}
\put(0,-.75){{\rm Figure~1. Bounded poset}}
\end{picture}
\end{center}
	
\vspace*{4mm}
	
We have
\[
\begin{array}{l|c|c|c|c|c|c|c}
             x & 0 &      a      &      b      &       c       &      d      &      e      & 1 \\
\hline
       x^\perp & P & \{0,b,c,e\} & \{0,a,c,d\} & \{0,a,b,d,e\} & \{0,b,c,e\} & \{0,a,c,d\} & 0 \\
\hline
x^{\perp\perp} & 0 &  \{0,a,d\}  &  \{0,b,e\}  &    \{0,c\}    &  \{0,a,d\}  &  \{0,b,e\}  & P
\end{array}
\]
Put $A:=\{0\}$ and $B:=\{0,b,c,e\}$. Then $A=\{1\}^\perp$ and $B=\{a\}^\perp$ and hence $A$ and $B$ are $\perp$-closed. If we define
\[
\begin{array}{r|c|c|c|c|c|c|c}
x    & 0 & a & b & c & d & e & 1 \\
\hline
f(x) & 0 & a & d & d & d & a & a \\
\hline
g(x) & 0 & 1 & 0 & 0 & 1 & 0 & 1
\end{array}
\]
then $f,g\in M(P)$ are orthogonally adjoint according to Theorem~\ref{th1}. Now put $A:=\{0,a,d\}$ and $B:=\{0,b,e\}$. Then $A=\{b,c\}^\perp$ and $B=\{a,c\}^\perp$ and hence $A$ and $B$ are $\perp$-closed. If we define
\[
\begin{array}{r|c|c|c|c|c|c|c}
x    & 0 &    a    &    b    &    c    &     d     &    e    &     1 \\
\hline
f(x) & 0 &    0    & \{c,d\} & \{a,c\} &     0     & \{a,c\} &  \{c,d\} \\
\hline
g(x) & 0 & \{b,c\} &    0    & \{c,e\} & \{b,c,e\} &    0    & \{b,c,e\}
\end{array}
\]
then $f,g\in Op(P)$ are orthogonally adjoint according to Theorem~\ref{th1}.
\end{example}

The converse of Theorem~\ref{th1} is also valid under some more or less natural conditions.

\begin{theorem}\label{th4}
Let $(P,\le,0)$ be a poset with $0$ and $f,g\in\big(Op(P)\big)\setminus\{0\}$ satisfying the following conditions:
\begin{align*}
& \text{If }x\in P\text{ and }f(x)\ne0\text{ then for all }y\in P, f(x)\wedge y=0\text{ is equivalent to }g(y)=0.\tag{2} \\	
& \text{If }y\in P\text{ and }g(y)\ne0\text{ then for all }x\in P, x\wedge g(y)=0\text{ is equivalent to }f(x)=0.\tag{3}
\end{align*}
Then $f$ and $g$ are orthogonally adjoint. Moreover, $A:=\{x\in P\mid f(x)=0\}$ and $B:=\{y\in P\mid g(y)=0\}$ are $\perp$-closed subsets of $P$ such that $f$ and $g$ are constructed in the way described in Theorem~\ref{th1}.
\end{theorem}

\begin{proof}
Let $a,b\in P$. \\
If $f(a)=0=g(b)$ then $f(a)\wedge b=0=a\wedge g(b)$. \\
If $f(a)=0\ne g(b)$ then $f(a)\wedge b=0=a\wedge g(b)$. \\
If $f(a)\ne0=g(b)$ then $f(a)\wedge b=0=a\wedge g(b)$. \\
If $f(a)\ne0\ne g(b)$ then $f(a)\wedge b\ne0\ne a\wedge g(b)$. \\
Hence $f$ and $g$ are orthogonally adjoint. Since $f,g\ne0$ we have $A,B\ne P$. Because $A=\big(g(y)\big)^\perp$ for all $y\in P\setminus B$
and $B=\big(f(x)\big)^\perp$ for all $x\in P\setminus A$, $A$ and $B$ are $\perp$-closed and because of (2) and (3), $f$ and $g$ are of the form described in Theorem~\ref{th1}.
\end{proof}

\begin{lemma}
Let $A$ and $B$ be proper subsets of $P$ and $f$ and $g$ be constructed as in Theorem~\ref{th1}. Then $f$ and $g$ satisfy {\rm(2)} and {\rm(3)} and the sets $A$ and $B$ defined in Theorem~\ref{th4} coincide with the corresponding sets in Theorem~\ref{th1}.
\end{lemma}

\begin{proof}
If there would exist some $a\in P\setminus A$ satisfying $f(a)=0$ then $B=\big(f(a)\big)^\perp=0^\perp=P$, a contradiction. Hence $A=\{x\in P\mid f(x)=0\}$. Analogously, $B=\{y\in P\mid g(y)=0\}$. Now for any $x\in P\setminus A$ we have $\big(f(x)\big)^\perp=B$, i.e.\ (2) holds. Analogously, (3) can be proved. The proof of the last assertion is trivial.
\end{proof}

\begin{example}
For the poset visualized in Fig.~1 and for $f$ and $g$ from Example~\ref{ex2}, {\rm(2)} and {\rm(3)} are satisfied and hence these assumptions are natural.
\end{example}

If $\mathbf P=(P,\le,0)$ is a poset with $0$ satisfying the Ascending Chain Condition and $f\in Op(P)$ is defined by $f(x):=x^{\perp\perp}$ for all $x\in P$ then $\mathbf P$ satisfies the identity
\[
\Max\bigg(\Big(f\big(\Max(x^\perp)\big)\Big)^\perp\bigg)\approx\Max(x^{\perp\perp})\tag{4}
\]
since
\[
\Max\bigg(\Big(f\big(\Max(x^\perp)\big)\Big)^\perp\bigg)\approx\Max\Big(\big(\Max(x^\perp)\big)^{\perp\perp\perp}\Big)\approx\Max\Big(\big(\Max(x^\perp)\big)^\perp\Big)\approx\Max(x^{\perp\perp})
\]
by Lemma~\ref{lem5}.

\begin{corollary}
Let $(P,\le,0)$ be a poset with $0$ satisfying the Ascending Chain Condition, let $f\in Op(P)$ be defined by $f(x):=x^{\perp\perp}$ for all $x\in P$ and define
\[
g(x):=\Max\bigg(\Big(f\big(\Max(x^\perp)\big)\Big)^\perp\bigg)
\]
for all $x\in P$. Then $f$ and $g$ are orthogonally adjoint.
\end{corollary}

\begin{proof}
According to (4) we have $g(x)=\Max(x^{\perp\perp})$ for all $x\in P$ and hence $f$ and $g$ are orthogonally adjoint according to {\rm(vi)} and {\rm(vii)} of Lemma~\ref{lem6} (cf.\ a similar formula for $g$ in Proposition~\ref{prop1}).
\end{proof}	

In what follows we are interested in the case when the value of orthogonally adjoint operators can be equal to $0$.

\begin{lemma}
Let $(P,\le,0,1)$ be a bounded poset, $f,g\in Op(P)$ and $a,b\in P$. Then the following hold:
\begin{enumerate}[{\rm(i)}]
\item $f(0),g(0)\subseteq\{0\}$,
\item $f(a)\subseteq\{0\}$ if and only if $g(1)\subseteq a^\perp$,
\item $a\le b$ and $f(b)\subseteq\{0\}$ implies $f(a)\subseteq\{0\}$.
\end{enumerate}
\end{lemma}

\begin{proof}
\begin{enumerate}[(i)]
\item $g(1)\subseteq P=0^\perp$ implies $f(0)\subseteq1^\perp=0$, and $f(1)\subseteq P=0^\perp$ implies $g(0)\subseteq1^\perp=0$.
\item $f(a)\subseteq0=1^\perp$ if and only if $g(1)\subseteq a^\perp$.
\item If $a\le b$ then everyone of the following statements implies the next one: $f(b)\subseteq0=1^\perp$, $g(1)\subseteq b^\perp$, $g(1)\subseteq a^\perp$, $f(a)\subseteq1^\perp=0$.
\end{enumerate}
\end{proof}

An important class of posets with $0$ is the class of so-called pseudocomplemented posets.

We say that $(P,\le,0)$ is {\em pseudocomplemented} if for each $a\in P$ the subset $a^\perp$ of $P$ has a greatest element $b$. This element is called the {\em pseudocomplement} of $a$ and is denoted by $a^*$. Clearly, $a\le a^{**}$ and $a^{***}=a^*$ for each $a\in P$, see e.g.\ \cite B.

One can see immediately from the definition that if $(P,\le,0,{}^*)$ is a pseudocomplemented poset then $f,g\in Op(P)$ are orthogonally adjoint if and only if
\[
f(x)\le y^*\text{ if and only if }g(y)\le x^*
\]
for all $x,y\in P$.

In the following we present a method for generating non-trivial pairs of orthogonally adjoint mappings in pseudocomplemented posets.

Let $(P,\le,0,{}^*)$ be a pseudocomplemented poset. We call an {\em element} $a$ of $P$ {\em $*$-closed} if $a^{**}=a$. This is equivalent to the fact that there exists some element $b$ of $P$ with $b^*=a$. Now we can state the following:

\begin{theorem}\label{th3}
Let $(P,\le,0,{}^*)$ be a pseudocomplemented poset, $a\in P$ and $\alpha$ and $\beta$ be $*$-closed elements of $P$. Construct $f,g\in M(P)$ as follows: If $a\le\alpha$ then put $f(a):=0$. If $a\not\le\alpha$ then let $f(a)$ be an element of $P$ satisfying $\big(f(a)\big)^*=\beta$. If $a\le\beta$ then put $g(a):=0$. If $a\not\le\beta$ then let $g(a)$ be an element of $P$ satisfying $\big(g(a)\big)^*=\alpha$. Then $f$ and $g$ are orthogonally adjoint.
\end{theorem}

\begin{proof}
For $a,b\in P$ the following are equivalent: \\
$f(a)\wedge b=0$, \\	
$a\le\alpha$ or $\Big(a\not\le\alpha$ and $b\le\big(f(a)\big)^*\Big)$, \\
$a\le\alpha$ or $(a\not\le\alpha$ and $b\le\beta)$, \\
$a\le\alpha$ or $b\le\beta$, \\
$b\le\beta$ or $a\le\alpha$, \\
$b\le\beta$ or $(b\not\le\beta$ and $a\le\alpha)$, \\
$b\le\beta$ or $\Big(b\not\le\beta$ and $a\le\big(g(b)\big)^*\Big)$, \\
$a\wedge g(b)=0$.
\end{proof}

\begin{example}\label{ex1}
Consider the pseudocomplemented poset $(P,\le,0,{}^*)$ depicted in Fig.~2:	

\vspace*{-3mm}

\begin{center}
\setlength{\unitlength}{7mm}
\begin{picture}(10,8)
\put(5,1){\circle*{.3}}
\put(3,3){\circle*{.3}}
\put(7,3){\circle*{.3}}
\put(1,5){\circle*{.3}}
\put(3,5){\circle*{.3}}
\put(7,5){\circle*{.3}}
\put(9,5){\circle*{.3}}
\put(5,7){\circle*{.3}}
\put(5,1){\line(-1,1)4}
\put(5,1){\line(1,1)4}
\put(5,7){\line(-2,-1)4}
\put(5,7){\line(-1,-1)2}
\put(5,7){\line(1,-1)2}
\put(5,7){\line(2,-1)4}
\put(3,3){\line(0,1)2}
\put(3,3){\line(2,1)4}
\put(7,3){\line(-2,1)4}
\put(7,3){\line(0,1)2}
\put(4.85,.3){$0$}
\put(2.35,2.85){$a$}
\put(7.4,2.85){$b$}
\put(2.35,4.85){$d$}
\put(.35,4.85){$c$}
\put(7.4,4.85){$e$}
\put(9.4,4.85){$h$}
\put(4.85,7.4){$1$}
\put(.1,-.75){{\rm Figure~2. Pseudocomplemented poset}}
\end{picture}
\end{center}

\vspace*{4mm}

The pseudocomplementation ${}^*$ is given by the following table:
\[
\begin{array}{l|c|c|c|c|c|c|c|c}
x   & 0 & a & b & c & d & e & h & 1 \\
\hline
x^* & 1 & h & c & h & 0 & 0 & c & 1
\end{array}
\]
Put $\alpha:=c$ and $\beta:=h$. Then $\alpha$ and $\beta$ are $*$-closed. If we define
\[
\begin{array}{r|c|c|c|c|c|c|c|c}
x    & 0 & a & b & c & d & e & h & 1 \\
\hline
f(x) & 0 & 0 & a & 0 & a & c & a & c \\
\hline
g(x) & 0 & h & 0 & h & b & b & 0 & b
\end{array}
\]
then $f$ and $g$ are orthogonally adjoint according to Theorem~\ref{th3}.
\end{example}

\begin{lemma}\label{lem1}
Let $(P,\le,0,{}^*)$ be a pseudocomplemented poset and $a\in P$. Then the following hold:
\begin{enumerate}[{\rm(i)}]
\item If $f,g\in Op(P)$ are orthogonally adjoint then the following are equivalent:
\[
f(a)\le a^*,\enspace f(a^{**})\le a^*,\enspace g(a)\le a^*,\enspace g(a^{**})\le a^*.
\]
\item if $f\in M(P)$ is defined by $f(x):=x^{**}$ for all $x\in P$ then $f$ is orthogonally adjoint to itself.
\end{enumerate}
\end{lemma}

\begin{proof}
\begin{enumerate}[(i)]
\item If $f,g\in Op(P)$ are orthogonally adjoint then the following are equivalent:
\[
f(a^{**})\le a^*,\enspace g(a)\le a^{***},\enspace g(a)\le a^*,\enspace f(a)\le a^*,\enspace f(a)\le a^{***},\enspace g(a^{**})\le a^*.
\]
\item This is a special case of (vii) of Lemma~\ref{lem6} since we have $a^\perp=[0,a^*]$ and $a^{\perp\perp}=[0,a^{**}]$.
\end{enumerate}
\end{proof}

\begin{lemma}
Let $\mathbf P=(P,\wedge,0)$ be a meet-semilattice with $0$ and $f$ an automorphism of $\mathbf P$. Then $f$ and $f^{-1}$ are orthogonally adjoint.	
\end{lemma}

\begin{proof}
For $a,b\in P$ the following are equivalent:
\[
f(a)\wedge b=0,\enspace f^{-1}\big(f(a)\wedge b\big)=0,\enspace f^{-1}\big(f(a)\big)\wedge f^{-1}(b)=0,\enspace a\wedge f^{-1}(b)=0.
\]
\end{proof}

\begin{example}\label{ex3}
Consider the bounded poset $(P,\le,0,1)$ from Fig.~1 and the following $f\in M(P)$:
\[
\begin{array}{l|c|c|c|c|c|c|c}
x    & 0 & a & b & c & d & e & 1 \\
\hline
f(x) & 0 & a & b & c & a & b & 0
\end{array}
\]
We have
\[
\begin{array}{l|c|c|c|c|c|c|c}
x^\perp                                & P & \{0,b,c,e\} & \{0,a,c,d\} & \{0,a,b,d,e\} & \{0,b,c,e\} & \{0,a,c,d\} & 0 \\
\hline
f^{-1}(x^\perp)                        & P & \{0,b,c,e\} & \{0,a,c,d\} & \{0,a,b,d,e\} & \{0,b,c,e\} & \{0,a,c,d\} & 0 \\
\hline
\big(f^{-1}(x^\perp)\big)^\perp        & 0 &  \{0,a,d\}  &  \{0,b,e\}  &    \{0,c\}    & \{0,a,d\}   &  \{0,b,e\}  & P \\
\hline
\big(f^{-1}(x^\perp)\big)^{\perp\perp} & P & \{0,b,c,e\} & \{0,a,c,d\} & \{0,a,b,d,e\} & \{0,b,c,e\} & \{0,a,c,d\} & 0
\end{array}
\]
One can see that $\big(f^{-1}(x^\perp)\big)^{\perp\perp}=f^{-1}(x^\perp)$ for all $x\in P$. That $f$ and $g$, where $g(x):=\big(f^{-1}(x^\perp)\big)^\perp$ for all $x\in P$, are orthogonally adjoint, follows from the next proposition.
\end{example}

We are going to show that if in a poset $(P,\le,0)$ with $0$, the mapping $f$ in $P$ satisfies the identity
\[
\big(f^{-1}(x^\perp)\big)^{\perp\perp}\approx f^{-1}(x^\perp)
\]
which is satisfied in Example~\ref{ex3} then an operator $g$ on $P$ forming an orthogonally adjoint pair together with $f$ can be constructed in explicit way, see the following result.

\begin{proposition}\label{prop1}
Let $(P,\le,0)$ be a poset with $0$ and $f\in M(P)$, assume $\big(f^{-1}(x^\perp)\big)^{\perp\perp}=f^{-1}(x^\perp)$ for all $x\in P$ and put $g(x):=\big(f^{-1}(x^\perp)\big)^\perp$ for all $x\in P$. Then $f$ and $g$ are orthogonally adjoint.
\end{proposition}

\begin{proof}
For $a,b\in P$ the following are equivalent:
\[
f(a)\subseteq b^\perp,\enspace a\subseteq f^{-1}(b^\perp),\enspace a\subseteq\big(f^{-1}(b^\perp)\big)^{\perp\perp},\enspace \big(f^{-1}(b^\perp)\big)^\perp\subseteq a^\perp,\enspace g(b)\subseteq a^\perp.
\]
\end{proof}

Observe that in Example~\ref{ex3} and Proposition~\ref{prop1}, $f^{-1}$ does not denote the mapping being inverse to $f$ ($f$ need not be bijective), but for each $A\subseteq P$, $f^{-1}(A)$ denotes the set $\{x\in P\mid f(x)\in A\}$.

\begin{lemma}\label{lem2}
Let $(P,\le,0,{}^*)$ be a pseudocomplemented poset, assume $f\in Op(P)$ and $g\in M(P)$ to be orthogonally adjoint and suppose that for all $a,b\in P$, $g(b)\le a$ is equivalent to $g(b)\le a^{**}$. Then $g$ is uniquely determined by $f$.
\end{lemma}

\begin{proof}
For $a,b\in P$ the following are equivalent:
\[
g(b)\le a,\enspace g(b)\le a^{**},\enspace f(a^*)\le b^*.
\]
Hence for every $y\in P$ the element $g(y)$ is the smallest element $x$ of $P$ satisfying $f(x^*)\le y^*$.
\end{proof}

Recall that an {\em order-automorphism} of a poset $(P,\le)$ is a bijection $f$ from $P$ to $P$ satisfying
\[
x\le y\text{ if and only if }f(x)\le f(y)
\]
for all $x,y\in P$.

\begin{theorem}\label{th2}
Let $\mathbf P=(P,\le,0,{}^*)$ be a pseudocomplemented poset and $f$ an order-automorphism of $\mathbf P$, assume that for all $a,b\in P$, $\big(f^{-1}(b^*)\big)^*\le a^*$ implies $a\le f^{-1}(b^*)$, define $g\in M(P)$ by $g(x):=\big(f^{-1}(x^*)\big)^*$ for all $x\in P$ and assume that for all $a,b\in P$, $g(b)\le a$ is equivalent to $g(b)\le a^{**}$. Then $g$ is the unique mapping in $P$ being orthogonally adjoint to $f$.
\end{theorem}

\begin{proof}
For $a,b\in P$ the following are equivalent:
\[
f(a)\le b^*,\enspace f(a)\le f\big(f^{-1}(b^*)\big),\enspace a\le f^{-1}(b^*),\enspace \big(f^{-1}(b^*)\big)^*\le a^*,\enspace g(b)\le a^*.
\]
Hence $f$ and $g$ are orthogonally adjoint. That $g$ is the unique mapping in $P$ such that $f$ and $g$ are orthogonally adjoint, follows from Lemma~\ref{lem2}.
\end{proof}

Note that if $\mathbf P$ satisfies the identity $x^{**}\approx x$ then the two assumptions in Theorem~\ref{th2} are automatically fulfilled.

Now we are interested in special mappings on posets, the so-called $\perp$-morphisms, and their relationship to orthogonal adjointness.

Let $\mathbf P=(P,\le,0)$ and $\mathbf Q=(Q,\le,0)$ be posets with $0$. An {\em $\perp$-morphism from $\mathbf P$ to $\mathbf Q$} is a mapping $f$ from $P$ to $Q$ satisfying
\[
x\perp y\text{ implies }f(x)\perp f(y)
\]
for all $x,y\in P$. An $\perp$-morphism from $\mathbf P$ to $\mathbf P$ is called an {\em $\perp$-morphism of $\mathbf P$}.

At first we show that these mappings are not exceptional.

\begin{lemma}
Let $\mathbf P=(P,\le,0)$ be a poset with $0$ and $f\in Op(P)$ satisfying $f(x)\le x$ for all $x\in P$. Then $f$ is an $\perp$-morphism of $\mathbf P$.
\end{lemma}

\begin{proof}
If $a,b\in P$, $a\perp b$, $c\in f(a)$, $d\in f(b)$ and $e\le c,d$ then $e\le c\le a$ and $e\le d\le b$ which together with $a\perp b$ yields $e=0$ showing $f(a)\perp f(b)$.
\end{proof}

\begin{example}
Consider the poset $\mathbf P=(P,\le,0)$ with $0$ visualized in Fig.~3:

\vspace*{-3mm}

\begin{center}
\setlength{\unitlength}{7mm}
\begin{picture}(6,6)
\put(3,1){\circle*{.3}}
\put(1,3){\circle*{.3}}
\put(5,3){\circle*{.3}}
\put(0,5){\circle*{.3}}
\put(2,5){\circle*{.3}}
\put(4,5){\circle*{.3}}
\put(6,5){\circle*{.3}}
\put(3,1){\line(-1,1)2}
\put(3,1){\line(1,1)2}
\put(1,3){\line(-1,2)1}
\put(1,3){\line(1,2)1}
\put(5,3){\line(-1,2)1}
\put(5,3){\line(1,2)1}
\put(2.85,.3){$0$}
\put(.35,2.85){$a$}
\put(5.4,2.85){$b$}
\put(-.15,5.4){$c$}
\put(1.85,5.4){$d$}
\put(3.85,5.4){$e$}
\put(5.85,5.4){$g$}
\put(-.3,-.75){{\rm Figure~3. Poset $\mathbf P$ with $0$}}
\end{picture}
\end{center}

\vspace*{4mm}

Let $f\in M(P)$ be defined as follows:
\[
\begin{array}{r|c|c|c|c|c|c|c}
   x & 0 & a & b & c & d & e & g \\
\hline
f(x) & 0 & d & e & a & a & b & b
\end{array}
\]
Then $f$ is an $\perp$-morphism of $\mathbf P$, but it is not monotone since $a\le c$, but $f(a)=d\not\le a=f(c)$, and does not satisfy $f(x)\le x$ for all $x\in P$ since $f(a)=d\not\le a$.
\end{example}

\begin{lemma}
Let $\mathbf P=(P,\le,0)$ be a poset with $0$ and $f$ a bijective $\perp$-morphism of $\mathbf P$ such that $f^{-1}$ is a $\perp$-morphism of $\mathbf P$, too. Then $f$ and $f^{-1}$ are orthogonally adjoint.
\end{lemma}

\begin{proof}
For $a,b\in P$ any of the following statements implies the next one:
\[
f(a)\wedge b=0,\enspace f^{-1}\big(f(a)\big)\wedge f^{-1}(b)=0,\enspace a\wedge f^{-1}(b)=0,\enspace f(a)\wedge f\big(f^{-1}(b)\big)=0,\enspace f(a)\wedge b=0.
\]
\end{proof}

Observe that any homomorphism $f$ from a meet-semilattice $\mathbf P=(P,\wedge,0)$ with $0$ to a meet-semilattice $\mathbf Q=(Q,\wedge,0)$ with $0$ is an $\perp$-morphism from $\mathbf P$ to $\mathbf Q$ since $a,b\in P$ and $a\wedge b=0$ imply $f(a)\wedge f(b)=f(a\wedge b)=f(0)=0$.

It is a question if for a couple of orthogonally adjoint mappings $f$ and $g$ in a poset $\mathbf P=(P,\le,0)$ with $0$ there exist extensions of $f$ and $g$, respectively, from $P$ to $D(\mathbf P)$ being still orthogonally adjoint. Unfortunately, this need not be the case in general as the following example shows.

\begin{example}
Consider the poset $\mathbf P=(P,\le,0)$ with $0$ depicted in Fig.~4:
\vspace*{-3mm}

\begin{center}
\setlength{\unitlength}{7mm}
\begin{picture}(4,4)
\put(2,1){\circle*{.3}}
\put(0,3){\circle*{.3}}
\put(2,3){\circle*{.3}}
\put(4,3){\circle*{.3}}
\put(2,1){\line(-1,1)2}
\put(2,1){\line(0,1)2}
\put(2,1){\line(1,1)2}
\put(1.85,.3){$0$}
\put(-.15,3.4){$a$}
\put(1.85,3.4){$b$}
\put(3.85,3.4){$c$}
\put(-.35,-.75){{\rm Figure~4. Poset $\mathbf P$}}
\end{picture}
\end{center}

\vspace*{4mm}

and $f\in M(P)$ defined by
\[
\begin{array}{l|c|c|c|c}
x    & 0 & a & b & c \\
\hline
f(x) & 0 & 0 & b & c
\end{array}
\]
Then $f$ is orthogonally adjoint to itself. This can be seen as follows: Let $x,y\in P$. \\
If $x,y\in\{0,b,c\}$ then $f(x)\wedge y=x\wedge y=x\wedge f(y)$. \\
If $x=a$ then $f(x)\wedge y=f(a)\wedge y=0\wedge y=0=a\wedge f(y)=x\wedge f(y)$. \\
If $y=a$ then $f(x)\wedge y=f(x)\wedge a=0=x\wedge0=x\wedge f(a)=x\wedge f(y)$.

The complete lattice $\mathbf D(\mathbf P)$ is visualized in Fig.~5:

\vspace*{-3mm}

\begin{center}
\setlength{\unitlength}{7mm}
\begin{picture}(6,6)
\put(3,1){\circle*{.3}}
\put(1,3){\circle*{.3}}
\put(3,3){\circle*{.3}}
\put(5,3){\circle*{.3}}
\put(3,5){\circle*{.3}}
\put(3,1){\line(-1,1)2}
\put(3,1){\line(0,1)4}
\put(3,1){\line(1,1)2}
\put(3,5){\line(-1,-1)2}
\put(3,5){\line(1,-1)2}
\put(2.85,.3){$0$}
\put(.35,2.85){$a$}
\put(3.4,2.85){$b$}
\put(5.4,2.85){$c$}
\put(2.85,5.4){$1$}
\put(-1.15,-.75){{\rm Figure~5. Complete lattice $\mathbf D(\mathbf P)$}}
\end{picture}
\end{center}

\vspace*{4mm}

Suppose there exist extensions $g,h\in M\big(D(\mathbf P)\big)$ of $f\in M(P)$ such that $g$ and $h$ are orthogonally adjoint. Then we conclude: \\
Since $1\wedge h(a)=1\wedge0=0$ we have $g(1)\wedge a=0$ and hence $g(1)\in\{0,b,c\}$. \\
Since $1\wedge h(b)=1\wedge b=b\ne0$ we have $g(1)\wedge b\ne0$ and hence $g(1)\in\{b,1\}$. \\
Since $1\wedge h(c)=1\wedge c=c\ne0$ we have $g(1)\wedge c\ne0$ and hence $g(1)\in\{c,1\}$. \\
This shows $g(1)\in\{0,b.c\}\cap\{b,1\}\cap\{c,1\}=\emptyset$ which is a contradiction. Therefore such extensions $g$ and $h$ of $f$ do not exist.
\end{example}

\begin{theorem}
Let $\mathbf P=(P,\le,0)$ be a poset with $0$ and $f,g\in M(P)$ be monotone and orthogonally adjoint. Moreover, assume that in $\mathbf D(\mathbf P)$ meet distributes over arbitrary joins, i.e.
\[
\left(\bigvee_{i\in I}A_i\right)\wedge A=\bigvee_{i\in I}(A_i\wedge A)
\]
whenever $A,A_i\in D(\mathbf P)$ for all $i\in I$. Then there exist orthogonally adjoint {\rm(}monotone{\rm)} extensions $\overline f$ and $\overline g$ of $f$ and $g$, respectively, from $P$ to $D(\mathbf P)$.
\end{theorem}

\begin{proof}
We define $\overline f,\overline g\in M\big(D(\mathbf P)\big)$ by
\begin{align*}
\overline f(A) & :=\bigvee_{x\in A}L\big(f(x)\big), \\
\overline g(A) & :=\bigvee_{x\in A}L\big(g(x)\big)
\end{align*}
for all $A\in D(\mathbf P)$. Of course, $\overline f$ and $\overline g$ are monotone. Now let $a\in P$ and $A,B\in D(\mathbf P)$. Since $f$ and $g$ are monotone, we have
\begin{align*}
\overline f\big(L(a)\big) & =\bigvee_{x\in L(a)}L\big(f(x)\big)=L\big(f(a)\big), \\
\overline g\big(L(a)\big) & =\bigvee_{x\in L(a)}L\big(f(x)\big)=L\big(g(a)\big),
\end{align*}
i.e.\ $\overline f$ and $\overline g$ are extensions of $f$ and $g$, respectively. Now
\begin{align*}
\overline f(A)\wedge B & =\left(\bigvee_{x\in A}L\big(f(x)\big)\right)\wedge\left(\bigvee_{y\in B}L(y)\right)=\bigvee_{x\in A,y\in B}\Big(L\big(f(x)\big)\cap L(y)\Big), \\
A\wedge\overline g(B) & =\left(\bigvee_{x\in A}L(x)\right)\wedge\left(\bigvee_{y\in B}L\big(g(y)\big)\right)=\bigvee_{x\in A,y\in B}\Big(L(x)\cap L\big(g(y)\big)\Big).
\end{align*}
Now the following are equivalent:

$\overline f(A)\wedge B=0$, \\
$L\big(f(x)\big)\cap L(y)=0$ for all $x\in A$ and all $y\in B$, \\
$f(x)\wedge y=0$ for all $x\in A$ and all $y\in B$.

Moreover, the following are equivalent:

$A\wedge\overline g(B)=0$, \\
$L(x)\cap L\big(g(y)\big)=0$ for all $x\in A$ and all $y\in B$, \\
$x\wedge g(y)=0$ for all $x\in A$ and all $y\in B$.

Since $f$ and $g$ are orthogonally adjoint in $\mathbf P$, we conclude that $\overline f(A)\wedge B=0$ and $A\wedge\overline g(B)=0$ are equivalent, i.e.\ $\overline f$ and $\overline g$ are orthogonally adjoint in $\mathbf D(\mathbf P)$.
\end{proof}

Let $\mathbf P=(P,\le,0)$ be a poset with $0$, $a\in P$ and $a_i\in P$ for all $i\in I$. An {\em ideal} of $\mathbf P$ is downward closed non-empty subset of $P$. Denote by $\Id(\mathbf P)$ the set of all ideals of $\mathbf P$ and by $I(a)$ the principal ideal $L(a)$ of $a$. The set $\Id(\mathbf P)$ is closed under arbitrary unions and intersections and hence $\BId(\mathbf P):=(\Id(\mathbf P),\subseteq)$ is a complete sublattice of $(2^P,\subseteq)$. Moreover, $x\mapsto I(x)$ is an embedding of $\mathbf P$ into $\BId(\mathbf P)$ which preserves arbitrary intersections, i.e.\ $a=\bigwedge\limits_{i\in I}a_i$ holds in $\mathbf P$ if and only $I(a)=\bigcap\limits_{i\in I}I(a_i)$ holds in $\BId(\mathbf P)$.

\begin{theorem}
Let $\mathbf P=(P,\le,0)$ be a poset with $0$ and $f,g\in M(P)$ be monotone and orthogonally adjoint. Then there exist orthogonally adjoint {\rm(}monotone{\rm)} extensions $\overline f$ and $\overline g$ of $f$ and $g$, respectively, from $P$ to $\Id(\mathbf P)$.
\end{theorem}

\begin{proof}
We define $\overline f,\overline g\in M\big(\Id(\mathbf P)\big)$ by
\begin{align*}
\overline f(I) & :=\bigcup_{i\in I}I\big(f(i)\big), \\
\overline g(I) & :=\bigcup_{i\in I}I\big(g(i)\big)
\end{align*}
for all $I\in\Id(\mathbf P)$. Of course, $\overline f$ and $\overline g$ are monotone. Now let $a\in P$ and $I,J\in\Id(\mathbf P)$. Since $f$ and $g$ are monotone, we have
\begin{align*}
\overline f\big(I(a)\big) & =\bigcup_{i\in I(a)}I\big(f(i)\big)=I\big(f(a)\big), \\
\overline g\big(I(a)\big) & =\bigcup_{i\in I(a)}I\big(g(i)\big)=I\big(g(a)\big),
\end{align*}
i.e.\ $\overline f$ and $\overline g$ are extensions of $f$ and $g$, respectively. Now
\begin{align*}
\overline f(I)\cap J & =\left(\bigcup_{i\in I}I\big(f(i)\big)\right)\cap\left(\bigcup_{j\in J}I(j)\right)=\bigcup_{i\in I,j\in J}\Big(I\big(f(i)\big)\cap I(j)\Big), \\
I\cap\overline g(J) & =\left(\bigcup_{i\in I}I(i)\right)\cap\left(\bigcup_{j\in J}I\big(g(j)\big)\right)=\bigcup_{i\in I,j\in J}\Big(I(i)\cap I\big(g(j)\big)\Big).
\end{align*}
Now the following are equivalent:
	
$\overline f(I)\cap J=0$, \\
$I\big(f(i)\big)\cap I(j)=0$ for all $i\in I$ and all $j\in J$, \\
$f(i)\wedge j=0$ for all $i\in I$ and all $j\in J$.
	
Moreover, the following are equivalent:
	
$I\cap\overline g(J)=0$, \\
$I(i)\cap I\big(g(j)\big)=0$ for all $i\in I$ and all $j\in J$, \\
$i\wedge g(j)=0$ for all $i\in I$ and all $j\in J$.
	
Since $f$ and $g$ are orthogonally adjoint in $\mathbf P$, we conclude that $\overline f(I)\cap J=0$ and $I\cap\overline g(J)=0$ are equivalent, i.e.\ $\overline f$ and $\overline g$ are orthogonally adjoint in $\BId(\mathbf P)$.
\end{proof}

Authors' addresses:

Michal Botur \\
Palack\'y University Olomouc \\
Faculty of Science \\
Department of Algebra and Geometry \\
17.\ listopadu 12 \\
771 46 Olomouc \\
Czech Republic \\
michal.botur@upol.cz

Ivan Chajda \\
Palack\'y University Olomouc \\
Faculty of Science \\
Department of Algebra and Geometry \\
17.\ listopadu 12 \\
771 46 Olomouc \\
Czech Republic \\
ivan.chajda@upol.cz

Helmut L\"anger \\
TU Wien \\
Faculty of Mathematics and Geoinformation \\
Institute of Discrete Mathematics and Geometry \\
Wiedner Hauptstra\ss e 8-10 \\
1040 Vienna \\
Austria, and \\
Palack\'y University Olomouc \\
Faculty of Science \\
Department of Algebra and Geometry \\
17.\ listopadu 12 \\
771 46 Olomouc \\
Czech Republic \\
helmut.laenger@tuwien.ac.at

\begin{thebibliography}9
\bibitem B
G.~Birkhoff, Lattice Theory. AMS, New York 1940.
\bibitem{CKL}
I.¸~Chajda, M.~Kola\v r\'ik and H.~L\"anger, Induced orthogonality in semilattices with $0$ and in pseudocomplemented lattices and posets. Order {\bf42} (2025), 577--592.
\bibitem{CL}
I.~Chajda and H.~L\"anger, Residuated mappings and homomorphisms in posets. Miskolc Math.\ Notes (submitted).
\bibitem{PPS}
J.~Paseka, J.K.~Putra and R.~Smolka, On $\tau$-based orthomodular dynamic algebras. Internat.\ J.\ Theoret.\ Phys.\ {\bf65} (2026), Paper No.\ 140, 37 pp.
\end{thebibliography}
\end{document}